\documentclass[11pt]{amsart}
\usepackage{amsmath,amssymb,mathtools,amsthm}
\usepackage[backend=biber,style=numeric,sorting=none]{biblatex}
\usepackage[hidelinks]{hyperref}
\usepackage{xcolor}
\usepackage{listings}
\usepackage{todonotes}
\usepackage{pdflscape}
\IfFileExists{Counterexample/counterexample.bib}
  {\addbibresource{Counterexample/counterexample.bib}}
  {\addbibresource{counterexample.bib}}

\newtheorem{theorem}{Theorem}
\newtheorem{lemma}[theorem]{Lemma}

\newtheorem{remark}[theorem]{Remark}
\newtheorem{example}{Example}
\newtheorem{corollary}{Corollary}

\DeclarePairedDelimiter{\norm}{\lVert}{\rVert}

\newcommand{\Deltas}{\boldsymbol{\Delta}}
\newcommand{\HHH}{\mathcal{H}}
\newcommand{\ip}[2]{\left\langle #1,#2\right\rangle}
\newcommand{\binvec}[2]{\begin{pmatrix}#1\\#2\end{pmatrix}}
\newcommand{\mat}[1]{\begin{pmatrix}#1\end{pmatrix}}

\newcommand{\tr}{\operatorname{tr}}

\title[Stable systems for which the OZF stability test fails]{Existence of stable Lur'e systems for which the O'Shea-Zames-Falb stability test fails}
\author{Andrey Kharitenko}
\thanks{\textsc{Optimization and Decision Intelligence Group, ETH Zurich.}\\
\textit{Email:} \texttt{akharitenko@ethz.ch}}
\date{}

\begin{document}
\begin{abstract}
In this note we show that the well-known stability test proposed by O'Shea, Zames and Falb is not necessary for robust stability of Lur'e systems with slope-restricted nonlinearities and thereby disprove a conjecture by J. Carrasco.
Specifically we construct a full-block multiplier which certifies stability of the example Lur'e interconnection from \cite{kharitenko2024exactness} that does not admit an OZF multiplier.
\end{abstract}


\maketitle

\section{Introduction}

A classical task in control theory is determining stability of Lur'e feedback interconnections with static nonlinearities.
This problem has its roots in the first half of 1940s \cite{Lure1944} and has given impetus to various subfields in control, including Popov's criterion, the powerful frequency-domain and quadratic-form methods by Yakubovich (including what is now
known as the Kalman--Yakubovich--Popov (KYP) lemma) \cite{Yakubovich1962,Kalman1963}, stability multipliers and integral quadratic constraints (IQC) theory by Megretski and Rantzer \cite{MegretskiRantzer1997}.
We point to \cite{liberzon2006essays,carrasco2016zames} for more references and a historical outlook.
One class of Lur'e systems that has received considerable attention is in which the static nonlinearity $\phi$ is slope-restricted, i.e.
\begin{align}
    \label{eq:slopeRestricted}
    0\leq \frac{\phi(x_1) - \phi(x_2)}{x_1 - x_2} \leq \kappa \text{\ \ for all\ \ } x_1,x_2 \in \mathbb{R}\,,
\end{align}
for some $0 < \kappa \leq \infty$ and its multi-dimensional analogues \cite{mancera2005all}.
Starting from O'Shea, Zames and Falb who introduced a powerful sufficient stability criterion \cite{ZamesFalb1968} (thereon called OZF-criterion), many researchers in the field of control have studied this configuration with the latter criterion or generalized it, 
including Jan Willems and R. Brockett \cite{willems1968rearrangement}, V. Yakubovich \cite{Yakubovich2000QuadraticCriterion}, C. Desoer and M. Vidyasagar \cite{desoer2009feedback,Vidyasagar1978}, N. Barabanov \cite{Barabanov2003Monotone}, M. Safonov \cite{safonov2000zames}, A. Megretski \cite{Megretski1993CircleCriterion}, D. Altshuller \cite{altshuller2012frequency}, C. Scherer \cite{fetzer2017absolute}, J. Carrasco \cite{carrasco2016zames} and many others.
Furthermore, previous stability tests for Lur'e systems with slope-restricted nonlinearities turned out to be special cases of the OZF-criterion and the latter was rediscovered multiple times in the signal processing and optimization literature, e.g. \cite{claasen1975frequency, kosyakin1983oscillations}.
Moreover, new applications, such as analysis of optimization algorithms \cite{lessard2016analysis}, systems with saturation \cite{heath2021multipliers} and neural network controllers \cite{yin2021stability,pauli2021linear}, have sparked a renewed interest in recent times \cite{zhang2022duality,su2023necessity}. 
The question whether this sufficient criterion is necessary for absolute stability as well, is therefore of great theoretical interest.
Due to the fact that no more powerful criterion has been found it has been conjectured to be the case by J. Carrasco \cite{su2023necessity}.
Results on the necessity of stability criteria are rare and more difficult to obtain than proofs of sufficiency.
Notable (counter-)examples include the proof that Popov-criterion is conservative \cite{pyatnitskii1973existence}, the necessity of the circle criterion within the class of static, time-varying sector-bounded nonlinearities \cite{Megretski1993CircleCriterion,Rump2001CircleConservatism} and the disproof of the exactness of the convex upper bound for the structured singular value \cite{Treil1999InfiniteGap}.
It was shown in \cite{kharitenko2024exactness} that the OZF stability test is equivalent to the stability of a lifted MIMO interconnection, but the exactness for the original SISO remained open. 
Here, motivated by \cite{vershik1987quadraticduality} and a finite-horizon counterexample provided by ChatGPT 5.5, we show that the OZF stability test is indeed conservative by developing a new stability test based on the copositive cone and showing that the interconnection considered in \cite{kharitenko2024exactness} for which the OZF stability test fails is actually stable.
In particular this explicitly disproves the mentioned conjecture due to J. Carrasco.


\section{Notation and Definitions}

\subsection{Notation}
This note uses the notation and definitions from \cite{kharitenko2024exactness}, which we repeat below. 
We define the sets $[N] = \{0,\ldots,N-1\}$ and $\ell_d^{2e} = (\mathbb{R}^d)^{\mathbb{N}_0}$.
For a multifunction $\phi:\mathbb{R}\rightrightarrows\mathbb{R}$, the associated static relation on signals is denoted by $\Delta_\phi$ on $\mathbb{R}^N$, i.e. $(z,w) \in \Delta_\phi$ iff $w_k \in \phi(z_k)$ for all $k \in [N]$.
The class of total \((0,\kappa)\)-slope-restricted relations with $0\in \phi(0)$ is
\[
 \Deltas_1(0,\kappa)
 :=
 \{\Delta_\phi\mid 0 \in \phi(0) \text{\ and \eqref{eq:slopeRestricted} holds}\}\,,
\]
where for $\kappa = \infty$ it is understood that $\phi$ is total and monotone \cite{kharitenko2024exactness}.
The set of doubly hyperdominant matrices is defined as 
\begin{align*}
    \HHH_N = 
    \{M \in \mathbb{R}^{N \times N} \mid M \boldsymbol{1}_N \geq 0\,,\; M^\top \boldsymbol{1}_N \geq 0\,,\; \forall i\neq j\,: e_i^\top M e_j \leq 0 \}\,.
\end{align*}
For $N = \infty$ it is understood that the row and column sums in this definition converge and $M$ defines a bounded operator on $\ell^2$.
Finally $C^*$ denotes the dual cone of $C$.
\\
\subsection{OZF stability test}
The O'Shea-Zames-Falb stability test for $N = \infty$ is now as follows:
\begin{theorem}[OZF stability test]
    Let $G$ be a stable system and suppose the interconnection between $G$ and $\Deltas_1(0,\kappa)$ is well-posed.
    Then the interconnection between $G$ and $\Deltas_1(0,\kappa)$ is robustly stable if there exists some \(\varepsilon>0\) and \(M\in\HHH_\infty\) such that 
    \begin{align}
    \label{eq:zamesFalbStabilityTest}
    \ip{\binvec{Gv}{v}}
    {\mat{
    0 & M^\top\\
    M & -\frac{1}{\kappa}(M+M^\top)
    }\binvec{Gv}{v}}
    \leq -\varepsilon \norm{v}_2^2
    \quad\text{for all }v\in\ell^2.
    \end{align}
\end{theorem}
Here well-posedness and robust stability the interconnection of $G$ and $\Deltas_1(0,\kappa)$ is defined in \cite{kharitenko2024exactness}.

In the following we mostly consider the monotone case $\kappa = \infty$, which is without loss of generality due to the loop transform $G \mapsto G - \frac{1}{\kappa}$ \cite{kharitenko2024exactness}

\section{O'Shea-Zames-Falb multipliers on a finite horizon}

First we prove the conservatism of the finite-horizon OZF multipliers.
Let $N \in \mathbb{N}$ and consider the set of all similarly ordered and unbiased sequences
\begin{align*}
    \mathcal{K}_{\operatorname{SO}}^N = \left\{\mat{x \\ y} \in \mathbb{R}^{2N} \mid (x_i-x_j)(y_i-y_j) \geq 0\,, x_i y_i \geq 0 \text{\ for\ } i,j \in [N]\right\}\,.
\end{align*}
as well as the set of all full-block multipliers
\begin{align*}
    \boldsymbol{\Pi}_{\operatorname{SO}}^N 
    = \{\Pi \in \mathbb{S}^{2N \times 2N} \mid \forall z \in \mathcal{K}_{\operatorname{SO}}^N\,:\; z^\top \Pi z \geq 0\} 
    = \{zz^\top \mid z \in \mathcal{K}_{\operatorname{SO}}^N\}^* \,.
\end{align*}
The corresponding subset of OZF multipliers is defined as
\begin{align*}
    \boldsymbol{\Pi}_{\operatorname{OZF}}^N = \left\{\Pi = \mat{0 & M^\top \\ M & 0} \mid M \in \mathcal{H}_N\right\}\,.
\end{align*}
In \cite{willems1968rearrangement} it was shown that $\boldsymbol{\Pi}_{\operatorname{OZF}}^N \subseteq \boldsymbol{\Pi}_{\operatorname{SO}}^N$.
This was improved in \cite{mancera2005all}, where it was established that
\begin{align*}
    \boldsymbol{\Pi}_{\operatorname{OZF}}^N = \boldsymbol{\Pi}_{\operatorname{SO}}^N \cap \left\{\mat{0 & M^\top \\ M & 0} \mid M \in \mathbb{R}^{N \times N}\right\}\,,
\end{align*}
i.e. $\boldsymbol{\Pi}_{\operatorname{OZF}}^N$ is the largest class of \emph{passivity multipliers} for the set $\mathcal{K}_{\operatorname{SO}}^N$.
Since trivially $\mathbb{S}_+^{2N \times 2N} \subseteq \boldsymbol{\Pi}_{\operatorname{SO}}^N$, we obtain the following inclusion\footnote{Note that the cone $\mathbb{S}_+^{2N \times 2N} + \boldsymbol{\Pi}_{\operatorname{OZF}}^N$ is closed, since it is a sum of a closed convex set and a polyhedral (convex) set and the recession condition holds \cite[Theorem 20.3]{rockafellar1970convex}}
\begin{align*}
    \mathbb{S}_+^{2N \times 2N} + \boldsymbol{\Pi}_{\operatorname{OZF}}^N \subseteq \boldsymbol{\Pi}_{\operatorname{SO}}^N\,.
\end{align*}
To the best of the author's knowledge, it was not known whether this inclusion is strict.
In the following we show that this is indeed the case.

\subsection{An elementary counterexample}
The following example was provided by ChatGPT 5.5 Pro Extended Thinking 
and is surprisingly simple.
\begin{example}
\label{ex:counterexampleSmall}
  For \(N=2\), define
  \[
  B=
  \mat{\frac{3}{4}&1\\-2&\frac{3}{4}},
  \qquad
  \Pi =
  \mat{I & B^\top\\B & I}.
  \]
  Then
  \[
  \Pi \in  \boldsymbol{\Pi}_{\operatorname{SO}}^2 \text{\ \ but\ \ } \Pi \notin \mathbb{S}_+^{4 \times 4} + \boldsymbol{\Pi}_{\operatorname{OZF}}^2\,.
  \]
\end{example}

\begin{proof}
First we show $\Pi \in \boldsymbol{\Pi}_{\operatorname{SO}}^2$.
Let $(x,y) \in \mathcal{K}_{\operatorname{SO}}^2$ be similarly ordered and unbiased, i.e. $(x, y) \in C_i \times C_i$ for some $i=1,\ldots,6$, where
\begin{align*}
 C_1&=\{z \in \mathbb{R}^2 \mid z_0\geq z_1\geq0\},\\
 C_2&=\{z \in \mathbb{R}^2 \mid z_1\geq z_0\geq0\},\\
 C_3&=\{z \in \mathbb{R}^2 \mid z_0\leq0\leq z_1\},
\end{align*}
and $C_4 = -C_1$, $C_5 = -C_2$, and $C_6 = -C_3$. 
Due to invariance of quadratic forms w.r.t. sign-change, it is enough to check the case where \(x\) and \(y\) both lie in one of
$C_1,C_2,C_3$. 
These cases can be parametrized as
\[
 C_1:\quad x=(a+b,a),\quad y=(c+d,c),
\]
\[
 C_2:\quad x=(a,a+b),\quad y=(c,c+d),
\]
\[
 C_3:\quad x=(-a,b),\quad y=(-c,d),
\]
with \(a,b,c,d\geq0\). Direct expansion gives
\[
 y^\top Bx\geq-\frac{5}{4}bc,\qquad
 y^\top Bx\geq-\frac{5}{4}ad,\qquad
 y^\top Bx\geq-bc.
\]
In the three cases, respectively,
\[
 \norm{x}\geq b,\quad \norm{y}\geq\sqrt2\,c;
 \qquad
 \norm{x}\geq\sqrt2\,a,\quad \norm{y}\geq d;
 \qquad
 \norm{x}\geq b,\quad \norm{y}\geq c.
\]
Since \(5/4<\sqrt2\), each case gives
\[
 y^\top Bx\geq-\norm{x}\norm{y}.
\]
Therefore
\[
 \mat{x\\y}^{\!\top}\Pi\mat{x\\y}
 =
 \norm{x}^2+\norm{y}^2+2y^\top Bx
 \geq
 \left(\norm{x}-\norm{y}\right)^2
 \geq0,
\]
and hence \(\Pi\in\boldsymbol{\Pi}_{\operatorname{SO}}^2\).
It remains to show that $\Pi \not \in \mathbb{S}_+^{4 \times 4} + \boldsymbol{\Pi}_{\operatorname{OZF}}^2$.
Suppose that
\[
 \Pi=
 T+\mat{0&M^\top\\M&0}\,,
\]
for some $T \in \mathbb{S}_+^{4 \times 4}$ and $M\in\mathcal H_2$.
Write
\[
 M=\mat{\alpha&-u\\-v&\delta},
 \qquad u,v\geq0.
\]
The row and column sum inequalities for \(M\in\mathcal H_2\) imply
\(\delta\geq u\) and \(\delta\geq v\). Moreover,
\[
 T=
 \mat{I&(B-M)^\top\\B-M&I} \geq 0,
\]
so the Schur complement gives
\[
1 
\geq \norm{B-M}^2
= (1+u)^2 + \left(\frac{3}{4}-\delta\right)^2
\]
Since \(u\geq0\), this forces $u=0$ and $\delta=\frac{3}{4}$ and therefore $v\leq3/4$. 
But then
\begin{align*}
    1 
    \geq \norm{B-M}
    \geq \lvert -2+v \rvert 
    > 1\,,
\end{align*}
a contradiction.
Thus $\Pi \not \in \mathbb{S}_+^{4 \times 4} + \boldsymbol{\Pi}_{\operatorname{OZF}}^2$.
\end{proof}

\subsection{Full block multipliers in terms of copositive matrices}
In this subsection, motivated by \cite{vershik1987quadraticduality} (see also \cite{noori2025reluqc,noori2026dynamic,biertumpfel2025exact} for recent works that also exploit copositivity characterizations of other multipliers), we will give a description of $\boldsymbol{\Pi}_{\operatorname{SO}}^N$ in terms of copositive matrices.
This will provide a way to construct tests different from the one provided by OZF multipliers (i.e. \eqref{eq:zamesFalbStabilityTest}) and that are potentially less conservative.
Before we present it, we will need to introduce some notation.
Let
\begin{align*}
    \operatorname{COP}_N
    = \{Q=Q^\top\in\mathbb{S}^{N\times N} \mid w^\top Qw\geq0\text{ for all }w\in\mathbb R_+^N\}.
\end{align*}
be the copositive cone.
For each permutation $\sigma \in \mathcal{S}_{[N+1]}$ on $[N+1]$ define the matrices
$L_\sigma\in \{-1,0,1\}^{N\times N}$, with rows $i=0,\ldots,N-1$ and
columns $k=1,\ldots,N$, by
\begin{align}
    \label{eq:transformationMatrices}
    (L_\sigma)_{ik}
    =
    \begin{cases}
        1,& \sigma^{-1}(N)< k \leq \sigma^{-1}(i)\,,\\
        -1,& \sigma^{-1}(i)< k \leq \sigma^{-1}(N)\,,\\
        0,&\text{otherwise}\,,
    \end{cases}
    \quad
    K_\sigma = \mat{L_\sigma & 0 \\ 0 & L_\sigma}\,.
\end{align}
Then we have the following description.
\begin{theorem}
    \label{thm:descriptionFullBlock}
    It holds that
    \begin{align*}
        \boldsymbol{\Pi}_{\operatorname{SO}}^N
        = \{\Pi \in \mathbb{S}^{2N \times 2N} \mid 
        K_\sigma^\top\Pi K_\sigma\in\operatorname{COP}_{2N}
        \; \forall \sigma \in \mathcal{S}_{[N+1]}\}\,.
    \end{align*}
\end{theorem}

\begin{proof}
    For each $\sigma \in \mathcal{S}_{[N+1]}$ let 
    \begin{align*}
        C_\sigma
        = \{x\in\mathbb R^N: x_{\sigma(0)}\leq x_{\sigma(1)}\leq\cdots\leq x_{\sigma(N)}\},
    \end{align*}
    where the value at the label $N$ is fixed as $x_N = 0$. 
    First we claim that $C_\sigma=L_\sigma\mathbb R_+^N$.
    Indeed, let $x\in C_\sigma$ and define the adjacent gaps
    \begin{align*}
        \alpha_k 
        = x_{\sigma(k)}-x_{\sigma(k-1)} 
        \geq 0\,,
    \qquad k=1,\ldots,N.
    \end{align*}
    Then for $\alpha = (\alpha_k)_{k=1}^N$ and $i \in 0,\ldots,N-1$ we have by telescoping
    \begin{align*}
        (L_\sigma \alpha)_i
        = \sum_{k=1}^N (L_\sigma)_{ik} \alpha_k
       &=
        \begin{cases}
            \sum_{k=\sigma^{-1}(N)+1}^{\sigma^{-1}(i)} (x_{\sigma(k)}-x_{\sigma(k-1)})\,, & \sigma^{-1}(N) < \sigma^{-1}(i)\,,\\
            \sum_{k=\sigma^{-1}(i)+1}^{\sigma^{-1}(N)} (x_{\sigma(k-1)}-x_{\sigma(k)})\,, & \sigma^{-1}(i) < \sigma^{-1}(N) 
        \end{cases} \\
        &= x_i\,.
    \end{align*}
    Thus $L_\sigma \alpha = x$ and hence $C_\sigma \subseteq L_\sigma\mathbb R_+^N$.
    The converse $C_\sigma \supseteq L_\sigma\mathbb R_+^N$ can be shown in a similar way and thus $C_\sigma = L_\sigma\mathbb R_+^N$.
    Next, clearly
    \begin{align}
        \label{eq:representationSO}
        \mathcal{K}_{\operatorname{SO}}^N = \bigcup_{\sigma \in \mathcal{S}_{[N+1]}} K_\sigma \mathbb{R}_+^{2N}\,,
    \end{align}
    where $K_\sigma$ is defined in \eqref{eq:transformationMatrices}.
    Indeed, $(x,y) \in \mathcal{K}_{\operatorname{SO}}^N$ means precisely that there exists a permutation $\sigma \in \mathcal{S}_{[N+1]}$ such that $(x,y) \in C_\sigma \times C_\sigma$ (i.e. $(x,y)$ can be ordered by the same permutation after extension with $x_N = y_N = 0$).
    By the previous derivation $C_\sigma \times C_\sigma = K_\sigma \mathbb{R}_+^{2N}$ and hence \eqref{eq:representationSO} follows.
    Now we have the following chain of equivalences
    \begin{align*}
        \Pi \in \boldsymbol{\Pi}_{\operatorname{SO}}^N
        &\Longleftrightarrow
        z^\top \Pi z \geq 0 \text{\ for all\ } z \in \mathcal{K}_{\operatorname{SO}} \\
        &\Longleftrightarrow
        w^\top K_\sigma^\top \Pi K_\sigma w \geq 0 \text{\ for all\ } w \in \mathbb{R}_+^{2N}\,,\; \sigma \in \mathcal{S}_{[N+1]} \\
        &\Longleftrightarrow 
        K_\sigma^\top \Pi K_\sigma \in \operatorname{COP}_{2N}\,,\;  \sigma \in \mathcal{S}_{[N+1]}\,,
    \end{align*}
    which finishes the proof.
\end{proof}

Note that while the characterization in Theorem~\ref{thm:descriptionFullBlock} is exact, checking if a matrix is copositive is co-NP-complete \cite{MurtyKabadi1987}.
Further, the test in Theorem~\ref{thm:descriptionFullBlock} entails testing all $(N+1)!$ matrices, which is computationally intractable even for moderate $N$.

\section{Stability tests on infinite horizon}
\label{sec:infiniteHorizonMultiplier}

In this section we use Theorem \ref{thm:descriptionFullBlock} to develop a general stability test for Lur'e interconnections with monotone nonlinearities.
First we recall the notion of hard IQCs.
Let $\Psi$ be any stable dynamical system defined by the realization
\begin{align}
    \label{eq:filter}
    \Psi = 
    \left[\begin{tabular}{c|c} $A_\Psi$ & $B_\Psi$ \\ \hline $C_\Psi$ & $D_\Psi$ \end{tabular} \right]
    = 
    \begin{cases*}
        \xi_{t+1}^\Psi = A_\Psi \xi_t^\Psi + B_\Psi w_t \\
        z_t = C_\Psi \xi_t^\Psi + D_\Psi w_t
    \end{cases*}\,, \quad
    \xi_0^\Psi = 0\,,
\end{align}
with inputs $w = (x, y) \in \ell^{2e} \times \ell^{2e}$, state $\xi^\Psi \in \ell_{n_\Psi}^{2e}$ and output $z \in \ell_p^{2e}$.
Further let $\Pi \in \mathbb{S}^{p \times p}$ be any symmetric matrix.
Then a relation $\Delta \subseteq \ell^{2e} \times \ell^{2e}$ is said to satisfy the hrd IQC defined by $(\Psi,\Pi)$ if 
\begin{align}
    \label{eq:hardIQC}
    \sum_{t=0}^{T-1} z_t^\top \Pi z_t \geq 0 \quad \text{\ for all\  } T \in \mathbb{N}_0\,,
\end{align}
for any $w = (x,y) \in \Delta$, where $z_t$ is the output of the filter $\Psi$ driven by $w$.
For a linear system $G$ with realization $(A,B,C,D)$ with $A \in \mathbb{R}^{n \times n}$ we define the product realization
\begin{align*}
    \Psi \mat{G \\ 1} 
    = \left[\begin{tabular}{cc|c} $A_\Psi$ & $B_\Psi \mat{C \\ 0}$ & $B_\Psi \mat{D \\ 1}$ \\ $0$ & $A$ & $B$ \\ \hline $C_\Psi$ & $D_\Psi \mat{C \\ 0}$ & $D_\Psi \mat{D \\ 1}$ \end{tabular} \right]
    =: \left[\begin{tabular}{c|c} $\mathcal{A}$ & $\mathcal{B}$ \\ \hline $\mathcal{C}$ & $\mathcal{D}$ \end{tabular} \right]
\end{align*}

The following well-known theorem \cite{seiler2014stability,scherer2018stability} shows how we can obtain stability using such IQCs.

\begin{theorem}[\cite{seiler2014stability,scherer2018stability}]
    \label{thm:stabilityHard}
    Let $\Delta$ satisfy the hard IQC defined by $(\Psi,\Pi)$.
    Let $G$ be a stable system and suppose that there exists some $\mathcal{X} \in \mathbb{S}^{(n_\Psi + n) \times (n_\Psi + n)}$ with
    \begin{align}
        \label{eq:stabilityHard}
        \mat{I & 0 \\ \mathcal{A} & \mathcal{B} \\ \mathcal{C} & \mathcal{D}}^\top \mat{-\mathcal{X} & 0 & 0 \\ 0 & \mathcal{X} & 0 \\ 0 & 0 & \Pi} \mat{I & 0 \\ \mathcal{A} & \mathcal{B} \\ \mathcal{C} & \mathcal{D}} < 0\,, \quad 
        \mathcal{X} \geq 0\,, \quad
        \tr(\mathcal{X}) = 1\,.
    \end{align}
    If the interconnection of $G$ and $\Delta$ is well-posed, then it is stable and the gain depends only on the negativity margin of the first LMI in \eqref{eq:stabilityHard}.
\end{theorem}

To see how we can use Theorem \ref{thm:stabilityHard} for our class of monotone nonlinearities $\boldsymbol{\Delta}_1(0,\infty)$, we recall the following, trivial lemma.

\begin{lemma}
    \label{lem:hardIQC}
    Let $\Psi = \Psi_N$ be defined by \eqref{eq:filter} with
    \begin{align*}
        A_\Psi = \mat{S & 0 \\ 0 & S}\,,\quad B_\Psi = \mat{e_1^{N-1} & 0 \\ 0 & e_1^{N-1}}\,,\; \\ 
        C_\Psi = \mat{0 & 0 \\ I_{N-1} & 0 \\ 0 & 0 \\ 0 & I_{N-1}}\,,\quad D_\Psi = \mat{e_1^N & 0 \\ 0 & e_1^N}\,,
    \end{align*}
    with $S \in \mathbb{R}^{(N-1) \times (N-1)}$ being the lower Jordan block with zero diagonal and $e_1^k \in \mathbb{R}^k$ first unit vector. 
    i.e. 
    \begin{align*}
        \xi_t^\Psi &= (x_{t-1},\ldots,x_{t-N+1},y_{t-1},\ldots,y_{t-N+1}) \in \mathbb{R}^{2(N-1)}\,, \\   
        z_t &= (x_t,\ldots,x_{t-N+1},y_t,\ldots,y_{t-N+1})  \in \mathbb{R}^{2N}\,.
    \end{align*}
    Then every $\Delta \in \boldsymbol{\Delta}_1(0,\infty)$ satisfies the hard IQC defined by $(\Psi,\Pi)$ for any $\Pi \in \boldsymbol{\Pi}_{\operatorname{SO}}^N$.
\end{lemma}

Combining Theorem \ref{thm:descriptionFullBlock}, Theorem \ref{thm:stabilityHard} and Lemma \ref{lem:hardIQC} we obtain
\begin{corollary}
    \label{cor:stabilityTestConcrete}
    Let $G$ be stable system and suppose that the interconnection between $G$ and every $\Delta \in \boldsymbol{\Delta}_1(0,\infty)$ is well-posed.
    Let $N \in \mathbb{N}$ and suppose that there exists some $\Pi \in \mathbb{S}^{2N \times 2N}$ and $\mathcal{X} \in \mathbb{S}^{(2(N-1) + n) \times (2(N-1) + n)}$ such that \eqref{eq:stabilityHard} and 
    \begin{align}
        \label{eq:hardSOConstraintExplicit}
        K_\sigma^\top \Pi K_\sigma \in \operatorname{COP}_{2N} \text{\ for all\ } \sigma \in \mathcal{S}_{[N+1]}\,.
    \end{align}
    Then the interconnection between $G$ and $\boldsymbol{\Delta}_1(0,\infty)$ is robustly stable.
\end{corollary}

In practice, since checking copositivity is intractable, we replace \eqref{eq:hardSOConstraintExplicit} by the well-known inner approximation of copositivity \cite{shaked2021copositive}, i.e. 
\begin{align}
    \label{eq:hardSOConstraintRelaxed}
    K_\sigma^\top \Pi K_\sigma \in \mathbb{S}_+^{2N \times 2N} + \mathbb{R}_+^{2N \times 2N} \text{\ for all\ } \sigma \in \mathcal{S}_{[N+1]}\,.
\end{align} 

Now we are ready to present our counterexample.

\begin{example}
    \label{ex:mainExample}
    Consider
    \begin{align}
        \label{eq:exampleSystem}
        G_1(z)=\frac{1.1z+0.6}{z^2+1.8z+0.9}
    \end{align}
    and a nonlinearity from $\boldsymbol{\Delta}_1(0,\kappa)$ for $\kappa > 0$. 
    In \cite{kharitenko2024exactness} it was shown that a Zames-Falb multiplier certifying stability via \eqref{eq:zamesFalbStabilityTest} only exists up to $\kappa \approx 1.8408$, while instability can be shown using the criterion from \cite{seiler2021construction} for $\kappa \geq 2.1471$.
    Using the criterion in Corollary~\ref{cor:stabilityTestConcrete} on the loop transformed plant $G = G_1 - \frac{1}{\kappa}$ and $\boldsymbol{\Delta}_1(0,\infty)$ we can obtain the following upper bounds on $\kappa$ via bisection
    \begin{center}
        \begin{tabular}{c|ccccc}
            $N$ & $2$ & $3$ & $4$ & $5$ & $6$ \\ \hline
            $\boldsymbol{\Pi}_{\operatorname{OZF}}^N$ & $0.2845$ & $0.8035$ & $1.7046$ & $1.7336$ & $1.8315$\\
            $\boldsymbol{\Pi}_{\operatorname{SO}}^N$ & $0.3317$ & $1.2304$ & $2.0784$ & $2.1466$ & $2.1466$
        \end{tabular}
    \end{center}
    where we use condition \eqref{eq:hardSOConstraintRelaxed} instead of \eqref{eq:hardSOConstraintExplicit} for the full test $\boldsymbol{\Pi}_{\operatorname{SO}}^N$.
    This (numerically) disproves a conjecture due to J. Carrasco \cite{su2023necessity}.
    To remove doubt about solvers operating in floating point arithmetic, in Appendix \ref{sec:counterexampleRational} we provide rational certificates $\mathcal{X} $ and $\Pi$ for $N = 4$ and $\kappa = 2$. 
\end{example}

\begin{remark}
    A potentially more powerful test would be given by constructing IQCs with terminal cost as described by Scherer and Veenman \cite{scherer2018stability,scherer2018iqc}, i.e. replacing the condition \eqref{eq:hardIQC} by
    \begin{align}
        \label{eq:IQCTerminalCost}
        \sum_{t=0}^{T-1} z_t^\top \Pi z_t + (\xi_T^\Psi)^\top Z \xi_T^\Psi \geq 0 \text{\ for all\ } T \geq 0\,,
    \end{align}
    with the additional variable $Z \in \mathbb{S}^{2(N-1) \times 2(N-1)}$ and replacing $\mathcal{X} \geq 0$ by the coupling constraint $\mathcal{X} - \operatorname{diag}(Z, 0) \geq 0$ in \eqref{eq:stabilityHard}.
    While such a parametrized IQC was given for the OZF multipliers in \cite{scherer2022dissipativity}, the argument there to relies on their particular row- and column structure and cannot be easily extended to general multipliers in $\boldsymbol{\Pi}_{\operatorname{SO}}^N$.
\end{remark}

\section{Conclusion}

In this paper we have presented a more general condition for stability of Lur'e systems with monotone nonlinearities in discrete time that is potentially less conservative than the well-known OZF multiplier stability test, thereby disproving a conjecture due to J. Carrasco.
Further work could include strengthening the conditions \eqref{eq:hardSOConstraintExplicit} to obtain more tractable stability tests.\\
\\
Acknowledgements: The author would like to thank Dennis Gramlich and Joaquin Carrasco for some helpful discussions.
\printbibliography

\newpage 
\appendix

\section{Rational certificates of Example \ref{ex:mainExample}}
\label{sec:counterexampleRational}

The following matrices satisfy \eqref{eq:stabilityHard} and \eqref{eq:hardSOConstraintRelaxed} for $N = 4$ and $\kappa = 2$ exactly:
\begin{align*}
    \Pi = \mat{\Pi_{11} & \Pi_{12} \\ \Pi_{12}^\top & \Pi_{22}}\,, \quad
    \mathcal{X} =  \mat{\mathcal{X}_{11} & \mathcal{X}_{12} \\ \mathcal{X}_{12}^\top & \mathcal{X}_{22}}\,,
\end{align*}
with
\begin{align*}
    \Pi_{11} &= 10^{-5}\mat{6649 & 7118 & -5403 & -9585 \\
7118 & 13071 & 4161 & -5012 \\
-5403 & 4161 & 26004 & 23038  \\
-9585 & -5012 & 23038 & 28676 }\,, \\
\Pi_{12} 
&= 10^{-5}\mat{11671 & -3863 & -3842 & -2906 \\ 
3293 & 13901 & -10052 & 2618 \\
-5398 & 8935 & 13416 & -14595 \\
4069 & -6155 & 12838 & -9745} \\
\Pi_{22}
&= 
10^{-5}\mat{2118 & -2819 & 3203 & -2926 \\
-2819 & 4048 & -4025 & 3177 \\
3203 & -4025 & 7693 & -8961 \\
-2926 & 3177 & -8961 & 11741}\,,\\
\mathcal{X}_{11}
&= 10^{-5}\mat{121924522 & 14295215 & 154603645 & -56572565 \\
14295215 & 77659559 & 7432961 & 84234444 \\
154603645 & 7432961 & 209670599 & -62831802 \\
-56572565 & 84234444 & -62831802 & 173928304 }\,, \\
\mathcal{X}_{12} 
&= 10^{-5}\mat{-24096819 & -44579430 & 6765287 & 3999883 \\
 -35944588 & -10548851 & -18015188 & 5972456 \\ 
-15108087 & -49736652 & 8781057 & 2512320 \\
-8952709 & 25155904 & -29967159 & 1478261} \\
\mathcal{X}_{22}
&= 10^{-5}\mat{28970192 & 16533564 & 4713848 & -4820492 \\
 16533564 & 19691549 & -2370424 & -2745880 \\
 4713848 & -2370424 & 5534872 & -790525 \\
 -4820492 & -2745880 & -790525 & 811419}
\end{align*}

\end{document}